\documentclass[11pt]{article}

\usepackage{graphicx,amssymb,amsfonts,latexsym,amsmath,amsthm,times}
\usepackage{epsfig}
\usepackage{fancyhdr}
\usepackage{color}
\usepackage[T1]{fontenc}
\usepackage[english]{babel}
\usepackage{epstopdf}
\DeclareGraphicsExtensions{.eps,.ps,.pdf,.bmp}

\newtheorem{thm}{Theorem}[section]
\newtheorem{cor}{Corollary}[section]

\newtheorem{lem}{Lemma}[section]
\newtheorem{prop}{Proposition}[section]
\theoremstyle{definition}
\newtheorem{defn}{Definition}[section]
\newtheorem{rem}{Remark}[section]

\begin{document}

% Title----------------------------------------------------------------

\title{
\begin{flushleft}
 \noindent {\small {\it  THE INTERNATIONAL CONFERENCE ON ALGEBRA, DISCRETE MATHEMATICS AND APPLICATIONS \textbf{(ICADMA22)},
29--30 September 2022, Faculty of Sciences A\"in Chock, Casablanca, Morocco. \vspace{0.4cm} \hrule \hrule \vspace{0.8cm}
 }}
\end{flushleft}
%%%%%%%%%%%%%%%%%%%%%%%%%%%%%%%%%%%%%%%%%%%%%%%%%%%%%%%%%%%%%%%%%%%%%%%%%%%%
%               PRIERE DE NE RIEN MODIFIER AU-DESSUS                       %
%%%%%%%%%%%%%%%%%%%%%%%%%%%%%%%%%%%%%%%%%%%%%%%%%%%%%%%%%%%%%%%%%%%%%%%%%%%%
%%%%%%%%%%%%%%%%%%%%%%%%%%%%%%%%%%%%%%%%%%%%%%%%%%%%%%%%%%%%%%%%%%%%%%%%%%%%
%                                                                          %
%        LES AUTEURS SONT PRIES DE FOURNIR  LES BONS ARGUMENTS             %
%       AUX COMMANDES titre, orateur, auteur, motsclefs CI-DESSOUS,        %
%                                                                          %
%%%%%%%%%%%%%%%%%%%%%%%%%%%%%%%%%%%%%%%%%%%%%%%%%%%%%%%%%%%%%%%%%%%%%%%%%%%%
%                                                                          %
% Remplacer "TITRE" par le titre de votre comunication                     %
%                                                                          %
% Remplacer "Auteur 1, 2, ..." par les Noms des auteurs                    %
%                                                                          %
% Remplacer "mail" par l'adresse électronique de correspondance            %
%                                                                          %
%%%%%%%%%%%%%%%%%%%%%%%%%%%%%%%%%%%%%%%%%%%%%%%%%%%%%%%%%%%%%%%%%%%%%%%%%%%%
\Large\bf \uppercase{On $p$-groups of maximal class}}

\author{Noureddine Snanou\thanks{Faculty of Sciences Dhar El Mahraz-Fez,  Sidi Mohamed Ben Abdellah University , E-mail: noureddine.snanou@usmba.ac.ma}}

\maketitle

%%%%%%%%%%%%%%%%%%%%%%%%%%%%%%%%%%%%%%%%%%%%%%%%%%%%%%%%%%%%%%%%%%%%%%%%%%%%
%                                                                          %
%                   DEBUT DE LA COMMUNICATION                              %
%                                                                          %
%%%%%%%%%%%%%%%%%%%%%%%%%%%%%%%%%%%%%%%%%%%%%%%%%%%%%%%%%%%%%%%%%%%%%%%%%%%%

{\bf Abstract.} \vskip2mm
Recall that a $p$-group of order $p^ {n} >p^ {3} $ is of maximal
class, if its nilpotency class is $n-1$. In this paper, we study the $p$-groups of maximal class. Furthermore, we introduce a subgroup of a $p$-group of maximal class called the fundamental subgroup. This group plays a fundamental role in the development of the general theory of $p$-groups of maximal class. As an application, we study some special class of finite $p$-groups of maximal class and exponent $p$.
\vspace{0.2cm}


\section{\bf Introduction}
A group of order $p^{n}>p^{3}$ and nilpotency class $n-1$, is said
to be of maximal class. These groups have been studied by various
authors \cite{C. R. I, W.A52}. But the main reference in the theory of $p$-groups of maximal class
is Blackburn's paper \cite{Bla1}. Other famous references for these $p$-groups are \cite{Ber1, GM02, Hup67}. The $p$-groups of maximal class with
an abelian maximal subgroup were completely classified by Wiman in \cite{W.A46}. More recently, it is proved that the finite $p$-group of maximal class can be determined by centralizers of some subgroups. For example, the theorem of
Suzuki ( \cite[Theorem III.14.23]{Hup67}, \cite[Proposition 1.8]{Ber1}) shows that a finite $p$-group $G$ is of maximal class if there
is a self centralizing subgroup of order $p^{2}$, and in \cite[Proposition 10.17]{Ber1}, it is showed that if $G$ is a finite $p$-group,
$B\leq G$ is a nonabelian subgroup of order $p^{3}$ such that $C_{G}(B)<B$, then $G$ is of maximal class. Some results about these groups that we present in this paper can be found in \cite{Ber1, Ber2, Ber3}. These results play a fundamental role in finite
$p$-group theory. Throughout this paper, we use the standard notation,
such as in \cite{Ber1}.

The paper is organized as follows: In the second section, we recall
some preliminaries of nilpotent groups. The third section covers the
basic material about $p$-groups of maximal class. In section $4$, we
introduce a characteristic subgroup of a $p$-group of maximal class called the
fundamental subgroup. This group plays a fundamental role in the
development of the general theory of $p$-groups of maximal class. In section 5, which is the end section of this work, we deal with some special class
of finite $p$-groups of maximal class and exponent $p$, namely $CGZ$-groups. More precisely, we prove that any $CGZ$-group of order $p^{n}$ and exponent $p$ (where $3\leq n\leq p$) admits a unique characteristic elementary abelian subgroup of index $p$.

\section{\bf Preliminaries}

 Let $G$ be a group and $H$ be a subgroup of $G$. An element $g\in G$ normalizes $H$ if $gHg^{-1}= H$. We call
$N_{G}(H)=\{g\in G | gHg^{-1}=H\}$ the normalizer of $H$ in $G$. An
element $g\in G$ centralizes $H$ if $ghg^{-1}=h$ for any $h\in H$.
We call $C_{G}(H) =\{g\in G | ghg^{-1}=h, \forall h\in H\}$ the
centralizer of $H$ in $G$. If $H=G$, then $Z(G)=C_{G}(G)$ is called
the center of $G$.

 Note that for a subgroup $H$ of $G$, $C_{G}(H)=G$
if and only if $H\leq Z(G)$. It is easy to see that, for any
subgroup $H$, $N_{G}(H)$ and $C_{G}(H)$ are subgroups of G with
$C_{G}(H)\leq N_{G}(H)$, and that if $H\leq G$ then $H\leq
N_{G}(H)$. The N/C-theorem \cite[Introduction, Proposition
12]{Ber1} asserts that if $H\leq G$, then the quotient group
$N_{G}(H)/C_{G}(H)$ is isomorphic to a subgroup of $Aut(H)$. In
particular, $G/Z(G)$ is isomorphic to a subgroup of $Aut(G)$.

 Let $G$ be a group. Set $Z_{0}(G)=\{1\},
Z_{1}(G)=Z(G)$. Suppose that $Z_{i}(G)$ has been defined for $i\leq
k$. Define $Z_{k+1}(G)$ as follows: $
Z_{k+1}(G)/Z_{k}(G)=Z(G/Z_{k}(G))$. The chain $\{1\}=Z_{0}(G)\leq
Z_{1}(G)\leq...\leq Z_{k}(G)\leq...$ is said to be the upper central
series of $G$. All members of that series are characteristic in $G$.

\begin{defn} For elements $x,y\in G$, their commutator $x^{-1}y^{-1}xy$ is written
as $[x,y]$. If $X,Y\subseteq G$, then $[X, Y]$ is the subgroup
generated by all commutators $[x,y]$ with $x\in X, y\in Y$.
\end{defn}

The lower central series $G=K_{1}(G)\geq K_{1}(G)\geq
K_{2}(G)\geq...$ of $G$ is defined as follows: $K_{1}(G)=G,
K_{i+1}(G)=[K_{i}(G),G], i>0$. All members of that series are
characteristic in $G$. We have $K_{i}(G)/K_{i+1}(G)\leq
Z(G/K_{i}(G))$. If $H\leq G$, then $K_{i}(H)\leq K_{i}(G)$ for all
$i$.

Since $[y,x]=[x,y]^{-1}$, we have $[Y,X]=[X,Y]$. We write
$[G,G]=G'$, the subgroup $G'$ is called the commutator (or derived
subgroup) of $G$. We also write $G^{0}=G$, $G'=G^{1}$. Then the
subgroup $G^{i+1}=[G^{i},G^{i}]$ is called the (i+1)th derived
subgroup of $G$, $i\geq 0$. The chain $G=G^{0}\geq G^{1}\geq ...\geq
G^{n}\geq...$ is called the derived series of $G$. All members of
this series are characteristic in $G$ and all factors
$G^{i}/G^{i+1}$ are abelian. The group $G$ is said to be solvable if
$G^{n}={1}$ for some $n$.

\begin{defn}
A group $G$ is said to be nilpotent if the upper central series of
$G$ contains $G$. In other words, $G$ is nilpotent of class $c$, if
$Z_{c}(G)=G$ but $Z_{c-1}(G)<G$, we write $c=cl(G)$. In particular,
the class of the identity group is $0$ and the class of a
nonidentity abelian group is $1$.
\end{defn}

\begin{lem} The following are equivalent:
\begin{enumerate}
  \item $G$ is nilpotent of class $c$.
  \item $G=K_{1}(G)>K_{2}(G)>K_{3}(G)>...>K_{c+1}(G)=\langle1\rangle$.
  \item $\langle1\rangle=Z_{0}(G)<Z_{1}(G)<...<Z_{c-1}(G)<Z_{c}(G)=G$.
\end{enumerate}
\end{lem}

\begin{thm}\label{NIL}
Let $G$ be a p-group of order $p^{m}\geq p^{2}$. Then:
\begin{enumerate}
\item $G$ is nilpotent of class at most $m-1$.
\item If $G$ has nilpotency class $c$ then $|G:Z_{c-1}(G)|\geq p^{2}$.
\item The maximal subgroups of $G$ are normal and of index $p$.
\item $|G:G'|\geq p^{2}$.
\end{enumerate}
\end{thm}

\begin{cor}\label{CORNIL} Let $G$ be a p-group and let $N$ be a normal subgroup of $G$
of index $p^{i}\geq p^{2}$. Then $K_{i}(G)\leq N$.
\end{cor}

\begin{proof} The group $G/N$ has order $p^{i}\geq p^{2}$. It follows from
part $(1)$ of the Theorem \ref{NIL} that $G/N$ has class $\leq i-1$
and consequently $K_{i}(G/N)=\overline{1}$. Since
$K_{i}(G/N)=K_{i}(G)N/N$, this proves that $K_{i}(G)\leq N$.
\end{proof}

\section{$p$-groups of maximal class }

Recall that a group of order $p^{m}$ is of maximal class, if
$cl(G)=m-1>1$. Of course, any group of order $p^{2}$ and any
nonabelian p-group of order $p^{3}$ have maximal class. If $G$ is a
group of maximal class and of order $p^{m}$, then the lower and
upper central series of $G$ are:
$$G=K_{1}(G)>K_{2}(G)>K_{3}(G)>...>K_{m}(G)=\langle1\rangle,$$
and
$$\langle1\rangle=Z_{0}(G)<Z_{1}(G)<...<Z_{m-2}(G)<Z_{m-1}(G)=G.$$
The two series are the same, and all members of these series are
characteristic in $G$. The sections are all of order $p$, except the first which has order $p^{2}
$ and is not cyclic. Clearly $G/K_{i}(G)$ is also of maximal class for $4\leq i\leq m-1$.

\begin{prop}\label{maxi}
Let $G$ be a $p$-group of maximal class and order $p^{m}$. Then:
\begin{enumerate}
  \item $|G:G'|=p^{2}$, $|Z(G)|=p$ and $|K_{i}(G):K_{i+1}(G)|=p$ for $2\leq i\leq m-1$.

  \item \label{maxim}If $1\leq i<m-1$, then $G$ has only one normal subgroup of order
$p^{i}$. More precisely, if $N$ is a normal subgroup of $G$ of index
$p^{i}\geq p^{2}$, then $N=K_{i}(G)$.
  \item $G$ has $p+1$ maximal subgroups.
\end{enumerate}
\end{prop}
\begin{proof}
(1) We have that
$$p^{m}=|G|=|G:G'|\prod_{i=2}^{m-1}|K_{i}(G):K_{i+1}(G)|$$. Now it
suffices to observe that $|G:G'|\geq p^{2}$, by theorem \ref{NIL},
and that $|K_{i}(G):K_{i+1}(G)|\geq p$ for $2\leq i\leq
m-1$. \\

(2) Let $N$ be any normal subgroup of $G$ of index $p^{i}$ with
$0\leq i\leq m$.If $i=0$ or $1$ then $N=K_{1}(G)$ or $N$ is maximal
in $G$. Otherwise $i\geq 2$ and $K_{i}(G)\leq N$ by Corollary
\ref{CORNIL}. Since $|G:K_{i}(G)|=p^{i}$, we conclude that
$N=K_{i}(G)$. \\

(3) As $G/K_{2}(G)$ has exponent $p$, the Frattini subgroup
$\Phi(G)=K_{2}(G)$. Hence $G/\Phi(G)$ has order $p^{2}$ and can be
regarded as a vector space over $\mathbb{F}_{p}$ of dimension $2$.
This vector space has $p+1$ subspaces of dimension $1$ and these
correspond to the maximal subgroups of $G$.
\end{proof}

\begin{rem}\
\begin{enumerate}
\item (The converse of the proposition \ref{maxi}(2)) If $G$ is a noncyclic group of order
$p^{m} > p^{2}$ containing only one normal subgroup of order $p^{i}$
for each $1\leq i<m-1$, then it is of maximal class \cite[Lemma 9.1]{Ber1}.
\item A p-group $G$ of order $p^{n}$ has exactly $n+p-1$ nontrivial normal subgroups if and only if it is of maximal class \cite[exercie 0.30]{Ber1}.
\end{enumerate}
\end{rem}

Suppose that a p-group $G$ has only one normal subgroup $T$ of index
$\geq p^{p + 1}$. If $G/T$ is of maximal class so is $G$
(see \cite[Theorem 12.9]{Ber1}). Conversely, let $G$ be a p-group of
maximal class. If $N$ is a normal subgroup of $G$ of index $\geq
p^{2}$, then $G/N$ has also maximal class. Indeed, since the class
of $G/K_{i}(G)$ is $i-1$ Whenever $2\leq i\leq m$, then the result
follows immediately from proposition \ref{maxi}.

\begin{prop} Let $G$ be a nonabelian group of order $p^{m} > p^{p}$. If $G$ contains only
one normal subgroup of index $p^{k}$ for any $k\in\{2, . . ., p +
1\} $, then it is of maximal class.
\end{prop}
\begin{proof} Obviously, $|G : G'|= p^{2}$ hence $ d(G) = 2 $. Assume that $G$
is not of maximal class, then $m > p + 1$. Let $T < G'$ be
G-invariant of index $p^{p + 1}$ in $G$, then $G/T$ is of maximal
class. As, by hypothesis, $T$ is the unique normal subgroup of index
$p^{p + 1}$ in $G$, then $G$ is of maximal class.
\end{proof}

\begin{prop}\label{M.Suzuki}(M.Suzuki). Let $G$ be nonabelian $p$-group. If $A<G$ of order
$p^{2}$ is such that $C_{G}(A)=A$, then $G$ is of maximal class.
\end{prop}
\begin{proof} We use induction on $|G|$. Since $p^{2}\nmid|Aut(A)|$ then, by
N/C-theorem, $N_{G}(A)$ is nonabelian of order $p^{3}$. As $A<G$,
then $ Z(G)<A $ and $|Z(G)|=p$. Obviously, $C_{G/Z(G)}(A/Z(G))=
N_{G/Z(G)}(A/Z(G))$. Since $ C_{G/Z(G)}(N_{G}(A)/Z(G))\leq
C_{G/Z(G)}(A/Z(G))=N_{G}(A)/Z(G)$ is of order $p^{2}$ then, by
induction, $G/Z(G)$ is of maximal class so $G$ is also of maximal
class since $|Z(G)|=p$.
\end{proof}

 \begin{cor} A $p$-group $G$ is of maximal class if and only if $G$ has an element with centralizer of
order $p^{2}$.
 \end{cor}

\begin{prop} \label{CGB}Let $G$ be a $p$-group, $B\leq G$ nonabelian of order $p^{3}$ and $C_{G}(B)< B$. Then
$G$ is of maximal class.
\end{prop}
\begin{proof} Assume that $|G|\geq p^{4}$ and the proposition has been proved
for groups of order $<|G|$ . It is known that a Sylow $p$-subgroup of
$Aut(B)$ is nonabelian of order $p^{3}$. Now, $C_{G}(B)=Z(B)=Z(G)$.
Therefore, by N/C-Theorem, $N_{G}(B)/Z(G)$ is nonabelian of order
$p^{3}$. If $x\in G-C_{G}(B)$ centralizes $N_{G}(B)/Z(G)$, then $x$
normalizes $B$ so $x\in N_{G}(B)$, a contradiction. Thus,
$C_{G}(N_{G}(B)/Z(G))<N_{G}(B)/Z(G)$ so, by induction, $G/Z(G)$ is
of maximal class. Since $|Z(G)|=p$, we are done.\end{proof}

\begin{prop}\label{NORM} Let $G$ be a $p$-group. If $G$ has a subgroup $H$ such
that $N_{G}(H)$ is of maximal class, then it is of maximal class.
\end{prop}
\begin{proof} Assume that $|N_{G}(H)| > p^{3}$ (otherwise, $C_{G}(H)=H$ and
$G$ is of maximal class, by Proposition \ref{M.Suzuki}). We use
induction on $|G|$. One may assume that $N_{G}(H)< G$, then $H$ is
not characteristic in $N_{G}(H)$ so by Proposition \ref{maxi}, we
have $|N_{G}(H):H|=p$ hence $|H|>p$. As $|Z(N_{G}(H))|=p$ and
$Z(G)<N_{G}(H)$, we get $Z(G)=Z(N_{G}(H))$ so $|Z(G)|=p$ and
$Z(G)<H$. Then $N_{G/Z(G)}(H/Z(G))=N_{G}(H)/Z(G)$ is of maximal
class, so $G/Z(G)$ is also of maximal class by induction. Since
$|Z(G)|=p$, then $G$ is of maximal class.
\end{proof}

\begin{lem}
Let $G$ be a $p$-group and let $N\lhd G$ be of order $> p$. Suppose
that $G/N$ of order $> p$ has cyclic center. If $R/N\lhd G/N$ is of
order $p$ in $G/N$, then $R$ is not of maximal class.
\end{lem}
\begin{proof} Let $T$ be a G-invariant subgroup of index $p^{2}$ in $N$. Then
$R\leq C_{G}(N/T)$ so $R/T$ is abelian of order $p^{3}$, and we
conclude that $R$ is not of maximal class.
\end{proof}

\begin{prop}\label{index}Let $A < G$ be of order $> p$. If all subgroups of $G$ containing $A$
as a subgroup of index $p$ are of maximal class, then $G$ is also of
maximal class.
\end{prop}
\begin{proof} Set $N=N_{G}(A)$. In view of proposition \ref{NORM} and
hypothesis, one may assume that $|N:A|>p$ (otherwise, there is
nothing to prove). Let $D<A$ be N-invariant of index $p^{2}$ ( $D$
exists since $|A| > p$). Set $C=C_{N}(A/D)$, then $C > A$. Let $F/A
\leq C/A$ be of order $p$, then $F$ is not of maximal class, a
contradiction.
\end{proof}

Let $H < G$ be of index $> p^{k}$, $k > 1$. If all subgroups of $G$
of order $p^{k}|H|$, containing $H$, are of maximal class, then $G$
is also of maximal class. Indeed, let $H < M < G$, where $|M :H|=
p^{k}-1$. Then all subgroups of $G$ containing $M$ as a subgroup of
index $p$, are of maximal class. Now the result follows from
Proposition \ref{index}.

\begin{prop}
Let $G$ be a p-group of maximal class and order $p^{m}$, $p > 2$, $m
> 3$, and let $N\lhd G$ be of index $p^{3}$. Then $exp(G/N)=p$.
\end{prop}
\begin{proof} Assume that this is false. Let $T$ be a G-invariant subgroup of
index $p$ in $N$. By hypothesis, $G/N$ has two distinct cyclic
subgroups $C/N$ and $Z/N$ of order $p^{2}$. Then $C/N\cap
Z/N=Z(G/T)$ and $G/T$ is not of maximal class, a
contradiction.
\end{proof}

Let $A$ be an abelian subgroup of index $p$ of a nonabelian p-group
$G$. By \cite[lemma 1.1]{Ber1}, we have $|G|=p|G'||Z(G)|$. Hence,
we have the following proposition.

\begin{prop}\label{AGG}
Suppose that a nonabelian $p$-group $G$ has an abelian subgroup $A$ of
index $p$. If $|G:G'|=p^{2}$, then $G$ is of maximal class.
\end{prop}
\begin{proof} We proceed by induction on $|G|$. One may assume that
$|G|>p^{3}$. We have $|Z(G)|=\frac{1}{p}|G:G'|=p$, so $Z(G)$ is a
unique minimal normal subgroup of $G$. Since $Z(G)<G'$, then wa have
$|G/Z(G):G'/Z(G)|=|G:G'|=p^{2}$. Hence, the quotient group $G/Z(G)$
is of maximal class by induction, and the result follows since
$|Z(G)|=p$.
\end{proof}

The previous result also holds if $G$ contains a subgroup of maximal
class and index $p$ (see \cite[Theorem 9.10]{Ber1}).

\begin{prop}
Let $G$ be a p-group of order $p^{4}$. Show that $G$ is of maximal
class if and only if $|G:G'|=p^{2}$.
\end{prop}
\begin{proof} Let $|G:G'|=p^{2}$. We have to prove that $G$ is
of class $3$. Assume that $cl(G)=2$. It follows from \cite[Lemma 65.1]{Ber2} that $G$ contains a nonabelian subgroup $B$ of order
$p^{3}$. By proposition \ref{CGB}, we have $G=BZ(G)$, then
$|G:G'|=p^{3}$, a contradiction.\end{proof}

\begin{prop}
Let $G$ be a $p$-group of order $p^{4}$ and exponent $p$. Prove that
if $G$ has no nontrivial direct factors then it is of maximal class.
\end{prop}
\begin{proof} Let $H$ be an $A_{1}$-subgroup of $G$; then
$|H|=p^{3}$. If $Z(G)\nleq H$, then $G=H\times C$, where $C<Z(G)$ is
of order $p$ such that $C\nleq H$. Thus, $Z(G)<H$ so $Z(G)$ is of
order $p$. Since $C_{G}(H)=Z(G)$, we get $C_{G}(H)<H$. Then, by
proposition \ref{CGB}, $G$ is of maximal class.\end{proof}

\begin{prop}
If $G$ is of order $p^{4}$ and exponent $p$. Prove that if $d(G)=2$
then $G$ is of maximal class.
\end{prop}
\begin{proof} Obviously, the group $G$ has an abelian subgroup
of index $p$. Since $G'=\Phi(G)$ and, by hypothesis, $|G:G'|=p^{2}$,
the result follows from Proposition \ref{AGG}. \end{proof}

\section{Fundamental subgroup of a $p$-group of maximal class}

\begin{lem}
Let $G$ be a $p$-group of maximal class and order $p^{m} $, $m>3$. For each $%
2\leq i\leq m-2 $, $M_{i}=C_{G}(K_{i}(G)/K_{i+2}(G)) $ is a maximal subgroup of $G$.
\end{lem}

\begin{proof} Indeed, $K_{i}(G)/K_{i+2}(G) $ is a noncentral normal subgroup of
order $p^{2}$ in ${G}/{K_{i+2}(G)}$. So by $N/C$-theorem, the quotient group $%
G\diagup M_{i}$ is isomorphic to a subgroup of $\ Aut(K_{i}(G)\diagup
K_{i+2}(G))$. But, a $p$-sylow subgroup of $\ Aut(K_{i}(G)\diagup K_{i+2}(G))$
has order $p$. Thus, $\left\vert G,M_{i}\right\vert =p$.
\end{proof}

The subgroup $M_{2}=C_{G}(K_{2}(G)/K_{4}(G)) $ plays distinguished role in
what follows; we denote it by $G_{1} $ and call the fundamental subgroup of $%
G$. In the following, if $G $ is of maximal class, then $G_{1}$ denotes
always the fundamental subgroup of $G$. We shall write $G_{i}$ instead $%
K_{i}(G)$ for all $i\geq 2$ when there is no possible confusion. We have $%
G\diagup G_{2}$ is elementary abelian of order $p^{2}$ and $\left[G_{i},G_{i+1}\right]=p$ for $1\leq i\leq n-1.$ Hence $\left\vert
G:G_{i}\right\vert =p^{i}$ for $1\leq i\leq n$.

\begin{prop}\label{charG}
Let $H$ and $K$ be two $p$-groups of maximal class. Let $\varphi:H\rightarrow K$ be an isomorphism, then $\varphi(H_{1})\subseteq K_{1}$.
\end{prop}
\begin{proof}  The subgroup $H_{1}$ is composed of the
elements $x\in H$ such that $[x,H_{2}]\leq H_{4}$. Let $x\in H_{1}$, since $\varphi(H_{2})=K_{2}$ and $\varphi(H_{4})=K_{4}$, it follows that
$$[\varphi(x),K_{2}]=[\varphi(x),\varphi(H_{2})]=\varphi([x,H_{2}])\leq \varphi(H_{4})=K_{4}.$$
Thus, $\varphi(x)\in K_{1}$. As required.
\end{proof}

\begin{cor}
Let $G$ be a $p$-group of maximal class. Then, $G_{1}$ is a characteristic subgroup of $G$.
\end{cor}
\begin{proof}
The corollary follows directly from the preceding proposition by taking $H=K$.
\end{proof}

\begin{rem}\label{Fund}\
\begin{enumerate}
  \item If $N$ is a normal subgroup of $G$ such that $|G/N|\geq p^{4}$, it
is clear from the definition that $(G/N)_{1}=G_{1}/N$.
  \item Let $G$ be a $p$-group of maximal class, $|G|> p^{4}$. Let $M$ be a maximal subgroup of $G$ and let $M_{1}$ be the
fundamental subgroup of $M$. Then $|G:M_{1}|=p^{2}$ and $M_{1}\lhd
G$ so $M_{1}=\Phi(G)<G_{1}$, and we get $M_{1}=G_{1}\cap M$.
\end{enumerate}
\end{rem}

\begin{lem}\cite[Theorem 9.6(e)]{Ber1}\label{Fund1}
Let $G$ be a group of maximal class and order $p^{m}$, $p > 2$, $m >
p+1$. Then the set of all maximal subgroups of $G$ is
$$\Gamma_{1}=\{M_{1}=G_{1}; M_{2}; ... ; M_{p+1}\}$$
Where $G_{1}$ is the fundamental subgroup of $G$, and the subgroups
$M_{2}; ... ; M_{p+1}$ are of maximal class.
\end{lem}

\begin{prop}
Let a $p$-group $G$ of maximal class have order $p^{m} > p^{p+1}$. If
$H<G$ is of order $> p^{p}$ and $H\nleq G_{1}$, then $H$ is of
maximal class.
\end{prop}
\begin{proof} We proceed by induction on $m$. If $m=p+ 2$, the result follows
from Lemma \ref{Fund1} (indeed, then all members of the set
$\Gamma_{1}-{G_{1}}$ are of maximal class). Now let $m > p + 2$ and
let $H\leq M\in \Gamma_{1}$, then $M$ is of maximal class (Lemma
\ref{Fund1}). The subgroup $M_{1}=M\cap G_{1}$ is the fundamental
subgroup of $M$ (Remark \ref{Fund}(2)). As $H\nleq M_{1}$, then $H$
is of maximal class, by induction, applied to the pair $H<M$.
\end{proof}

\begin{prop}
Let $G$ be of maximal class and order $>p^{p+1}$. If $H<G$ is of
order $>p^{2}$, then either $H\leq G_{1}$ or $H$ is of maximal.
class.
\end{prop}
\begin{proof} Let $|H|=p^{k}$. The result is known if $k=3$ \cite{Bla1}.
Assuming that $k>3$, we use induction on $k$. Then all maximal
subgroups of $H$ which $\neq H\cap G_{1}$ are of maximal class, by
induction. Then the set $\Gamma_{1}(H)$ contains exactly
$|\Gamma_{1}(H)|-1\neq 0(mod p^{2})$ members of maximal class so $H$
is of maximal class, by \cite[Theorem 12.12(c)]{Ber1}.
\end{proof}

\begin{prop}\cite[Exercise 9.28]{Ber1}
Let $G$ be a $p$-group of maximal class and order $>p^{p+1}$. Show the
following:
\begin{enumerate}
  \item $exp(G_{1})=exp(G)$.
  \item If $x\in G$ is of order $\geq p^{3}$, then $x\in G_{1}$.
\end{enumerate}
\end{prop}

\section{$CGZ$-group of exponent $p$}

Let $p$ be a prime number and $G$ be a finite nonabelian p-group. The
group $G$ is called a $CGZ$-group if and only if every nonabelian subgroup $H
$ of $G$ satisfies $C_{G}(H)=Z(H)$. In this section, we prove that any $%
CGZ$-group of order $p^{n}$ and exponent $p$ (where $3\leq n\leq p$)
admits a unique characteristic elementary abelian subgroup of index $p$.

\begin{prop}
Let $G$ be a $p$-group of order $p^{4}$ and exponent $p$. If $G$ has no nontrivial direct factors then it is a $CGZ$-group.
\end{prop}
\begin{proof} Indeed, if $G$ is minimal nonabelian then it is clearly a $CGZ$-group. Else, let $H$ be a proper nonabelian subgroup of $G$, then
$|H|=p^{3}$. Now, let $x\in C_G(H)$. If $x\notin H$ then $G=\langle x\rangle\cdot H$, a contradiction. So, $C_{G}(H)<H$ and then $C_{G}(H)=Z(H)$. As required.
\end{proof}

\begin{lem}\label{lemCGZ}
Let $G$ be a nonabelian $p$-group and let $N$ be a nonabelian subgroup of $G$. If $G$ is a $CGZ$-group, then $N$ is a $CNZ$-group.
\end{lem}

\begin{proof}
Let $H$ be a nonabelian subgroup of $N$. Since $G$ is a $CGZ$-group, it
follows that $C_{N}(H)=N\cap C_{G}(H)=N\cap Z(H)=Z(H)$. So $N$ is a $CNZ$-group.
\end{proof}

\begin{prop}\cite[Proposition 2.2]{XLC14}\label{CGZ3}
Let $G$ be a nonabelian $p$-group of exponent $p$, where $p\geq 3$. Then $G$ is a $CGZ$-group if and only if $G$ is a finite $p$-group of maximal
class, and there is an abelian subgroup of index $p$.
\end{prop}

\begin{prop}
\label{CGZ2} Let $G$ be a $p$-group of order $p^{n}> p^{4}$ and exponent $p$%
. Then $G$ is a $CGZ$-group if and only if all proper nonabelian subgroups
of $G$ have maximal  class.
\end{prop}

\begin{proof} Let $H$ be a nonabelian subgroup of $G$. If $G$
is a $CGZ$-group, by Lemma \ref{lemCGZ}, $H$ is a $CHZ$-group, then
it is of maximal class, by Proposition \ref{CGZ3}. Conversely, let $H$ be a minimal nonabelian subgroup of $G$,
by \cite[Lemma 136.2(ii)]{Ber3}, we have $|H|=p^{3}$. Now, if $x\in C_G(H)$ then $K=\langle x\rangle\cdot H$ is a nonabelian subgroup of $G$ and then it is of maximal class. But $Z(H)\leq Z(K)$ and $|Z(H)|=|Z(K)|=p$, so $x\in Z(H)$. Thus $C_G(H)\leq H$ and then $G$ is of maximal class, by Proposition \ref{CGB}. Let $R$ be a normal subgroup of order $p^{2}$. By $N/C$-theorem, we have $\left\vert G,C_{G}(R)\right\vert =p$. If $C_{G}(R)$ is not abelian, by hypothesis, it is of maximal class and then
$\left\vert Z\left( C_{G}(R)\right) \right\vert =p$. But $R\leq Z\left(
C_{G}(R)\right) $, a contradiction. So $C_{G}(R)$ is abelian and maximal in $%
G$. Therefore, by Proposition \ref{CGZ3}, $G$ is a $CGZ$-group.
\end{proof}

\begin{prop} \cite[Proposition 2.6]{XLC14}\label{CGZ4}
Let $G$ be a finite $p$-group of maximal class. If $G$ is a $CGZ$-group, then $G'$ is abelian.
\end{prop}

\begin{rem}
Let $G$ be a $CGZ$-group of order $p^{n}$ and exponent $p$, by Proposition \ref{CGZ4}, $G^{\prime }$ is abelian, then $%
G_{1}=C_{G}(G_{i}\diagup G_{i+2})$ with $1\leq i\leq n-2$ (see \cite[Theorem III.14.11]{Hup67}).
\end{rem}

Hence we have the following interesting corollary:

\begin{cor}
\label{CGZ}Let $G$ be a $CGZ$-group of order $p^{n}$ and exponent $p$, where
$p\geq 3$. Then $G$ possesses a unique characteristic elementary abelian
subgroup of index $p$. In this case, we have necessary $n\leq p$.
\end{cor}

\begin{proof} Let $K=C_{G}(Z_{2}(G))$. By proposition \ref%
{CGZ2}, we easily see that $K$ is an abelian subgroup of $G$. By the above,
we get $K=G_{1}$, then $K$ is the unique characteristic elementary abelian
subgroup of index $p$. Since $G$ has exponent $p$, it is a regular p-group (see \cite[Theorem 7.1(b)]{Ber1}). Therefore, by a theorem of Blackburn (\cite[Theorem III.14.21]{Hup67}) or \cite[Theorem 9.5]{Ber1}), we obtain $|G|\leq p^{p}$.\end{proof}

%%%%%%%%%%%%%%%%%%%%%%%%%%%%%%%%%%%%%%%%%%%%%%%%%%%%%%%%%%%%%%%%%%%%%%%%%%%%
%               REFERENCES                                                 %
%%%%%%%%%%%%%%%%%%%%%%%%%%%%%%%%%%%%%%%%%%%%%%%%%%%%%%%%%%%%%%%%%%%%%%%%%%%%
\bibliographystyle{plain}

\end{document}